\documentclass[12pt]{article}
\usepackage[utf8]{inputenc}
\usepackage[margin=1in]{geometry}
\usepackage{amsmath}
\usepackage{amssymb}
\usepackage{amsthm}
\usepackage{mathtools}
\usepackage{hyperref}
\usepackage[english]{babel}
\usepackage[style=alphabetic]{biblatex}
\usepackage{csquotes}
\usepackage{xcolor}
\usepackage{bbm}
\usepackage{tikz}
\usetikzlibrary{decorations.pathmorphing, decorations.pathreplacing, decorations.shapes}
\usepackage{multicol}

\newcommand{\Ac}{\mathcal{A}}

\newcommand{\1}{\mathbbm{1}}

\DeclareMathOperator{\diam}{diam}
\DeclarePairedDelimiter{\abs}{|}{|}

\DeclarePairedDelimiter{\ceil}{\lceil}{\rceil}

\newtheorem{theorem}{Theorem}[section]
\newtheorem*{theorem*}{Theorem}
\newtheorem{lemma}[theorem]{Lemma}
\newtheorem{proposition}[theorem]{Proposition}

\newtheorem{corollary}[theorem]{Corollary}
\theoremstyle{definition}

\newcommand{\hide}[1]{}

\title{On the Diameter of Finite Sidon Sets}
\author{Daniel Carter\thanks{Princeton University, email: dc65@princeton.edu}\and Zach Hunter\thanks{ETH, email: zach.hunter@math.ethz.ch}\and Kevin O'Bryant\thanks{City University of New York, College of Staten Island and The Graduate Center, email: kevin.obryant@csi.cuny.edu}}
\date{}

\begin{document}

\maketitle

\begin{abstract}
We prove that the diameter of a Sidon set (also known as a Babcock sequence, Golomb ruler, or $B_2$ set) with $k$ elements is at least $k^2-b k^{3/2}-O(k)$ where $b\le 1.96365$, a comparatively large improvement on past results. Equivalently, a Sidon set with diameter $n$ has at most $n^{1/2}+0.98183n^{1/4}+O(1)$ elements. The proof is conceptually simple but very computationally intensive, and the proof uses substantial computer assistance. We also provide a proof of $b\le 1.99058$ that can be verified by hand, which still improves on past results. Finally, we prove that $g$-thin Sidon sets (aka $g$-Golomb rulers) with $k$ elements have diameter at least $g^{-1} k^2 - (2-\varepsilon)g^{-1}k^{3/2} - O(k)$, with $\varepsilon\ge 0.02g^{-2}$.
\end{abstract}

\section{Introduction}
A Sidon set is a set of integers $\Ac$ that does not contain any solutions to
    \[a-b=c-d,\quad a,b,c,d\in \Ac\]
except for the trivial types $a=b,c=d$ and $a=c,b=d$. These sets are also called Babcock sequences~\cite{1953.Babcock,1976.Johannsen&Paulsen}, Golomb rulers~\cite{mg, dr} and $B_2$ sets~\cite{sidon,et}. All intervals in this paper are intervals of integers; for example, $[1,4)=\{1,2,3\}$.

The first question asked by Sidon was to bound $k=\abs{\Ac}$, subject to the constraints that $\Ac$ is a Sidon set contained in $[0,n)$~\cite{sidon}; we set $R(n)$ to be the maximum cardinality of a Sidon set contained in $[0,n)$. Equivalently, one can ask for a bound on $\diam(\Ac)\coloneqq \max\Ac-\min\Ac$ in terms of $k$. In 1938, Singer constructed Sidon sets with $k=q+1$ elements, where $q$ is any prime power, and diameter less than $k^2-k$~\cite{singer}. In 1941, Erd\H{o}s and Tur\'an proved an inequality that implies that the diameter of a $k$-element Sidon set must be at least $k^2-2k^{3/2}-O(k)$; equivalently, $R(n)\le n^{1/2}+n^{1/4}+O(1)$~\cite{et} (though in their paper the constant on $k^{3/2}$ or $n^{1/4}$ was not given explicitly). We define
\[ b_\infty \coloneqq \limsup_{k\to\infty}\frac{k^2-\diam(\Ac_k)}{k^{3/2}} \]
where for each $k$, $\Ac_k$ is a Sidon set with $k$ elements and minimum possible diameter; the results of Singer and Erd\H{o}s--Tur\'an show that $0\le b_\infty\le 2$.

While there have been hundreds of articles about Sidon sets (see~\cite{obryantbiblio} for an extensive bibliography), these bounds on $b_\infty$ remained unimproved until recently. In 1969, Lindstr\"om gave a different argument that gives the same bound as Erd\H{o}s--Tur\'an~\cite{lindstrom}. In 2021, Balogh--F\"uredi--Roy~\cite{bfr} combined this proof with the Erd\H{o}s--Tur\'an proof and obtained $R(n) <n^{1/2}+0.998n^{1/4}+O(1)$; equivalently, $b_\infty \le 1.996$. The improvement found by Balogh--F\"uredi--Roy involves playing the proofs of Erd\H{o}s--Tur\'an and Lindstr\"om against each other: a set that forces a part of the Erd\H{o}s--Tur\'an argument to be weak allows a part of the Lindstr\"om argument to become strong.

In 2022, the third author showed how such an improvement could be made using only the Erd\H{o}s--Tur\'an argument, both simplifying the proof that $b_\infty < 2$ and improving the upper bound on $b_\infty$ to at most $1.99405$ \cite{summerpaper}.

In the present work, we give an even simpler exploitation of the Erd\H{o}s--Tur\'an method, and prove in Section~\ref{sec:nwindow} that
\[ b_\infty \le 1.96365, \]
an order of magnitude larger improvement than previous results. While this method is logically simpler than previous two improvements, it is computationally much more involved, and a full verification uses substantial computer assistance. In light of this, in Section~\ref{sec:twowindow} we provide a human-verifiable proof that
\[ b_\infty \le 1.99058. \]
using a simplified version of the main result.

Finally, in Section~\ref{sec:thin}, we consider $g$-thin Sidon sets (also called $g$-Golomb rulers \cite{ggolomb}): a set $\Ac$ is a \textit{$g$-thin Sidon set} if for each $x\neq 0$ there are at most $g$ pairs $(a,b)\in \Ac^2$ with $x=a-b$. In~\cite{bfr}, Balogh--F\"uredi--Roy note that their method also gives an improved bound on the diameter of $g$-thin Sidon sets, but they didn't make the improvement explicit. Our method gives the bound that if $\Ac$ is a $g$-thin Sidon set, then
\[ \diam(\Ac)\ge \frac{1}{g}k^2 - \frac{2-\varepsilon}{g}k^{3/2}-O(k) \]
with $\varepsilon\ge \frac{1}{50g^2}$. Here, the constant $1/50$ is not optimized.

\subsection{The Erd\H{o}s--Tur\'an Sidon Set Equality}

Two inequalities are used in the original Erd\H{o}s--Tur\'an proof. One can name the slack in those inequalities $V,S$ (as done in~\cite{summerpaper}), and arrive at the Erd\H{o}s-Tur\'an Sidon Set Equality (hereafter ETSSE):

\begin{theorem}[ETSSE~\cite{summerpaper}]\label{thm:erdosturan}
Let $\Ac$ be a finite Sidon set and $T$ be a positive integer. Then
\[ \diam(\Ac) = \frac{\abs{\Ac}^2 T^2}{T(T+\abs{\Ac}-1) - (2S(\Ac,T)+V(\Ac,T))} - T \]
where we define $A_i^{(T)} = \abs{\Ac \cap [i-T,i)}$,
\[ S(\Ac,T) = \sum_{\substack{r=1\\r\not\in \Ac-\Ac}}^{T-1} (T-r),  \quad\text{ and }\quad
V(\Ac,T) = \sum_{i=\min(\Ac)+1}^{T+\max(\Ac)} \left( A_i^{(T)} - \frac{\abs{\Ac} T}{T+\diam(\Ac)}\right)^2. \]
\end{theorem}

In Section~\ref{sec:thin}, we generalize this result to $g$-thin Sidon sets.

The trivial bounds $S(\Ac,T)\ge0,V(\Ac,T)\ge0$, and setting $T=\ceil{k^{3/2}}$ allow one to quickly get $\diam(\Ac)\ge k^2-2k^{3/2}-O(k)$, as is done explicitly in~\cite{summerpaper}; this is not fundamentally different from the 1941 work~\cite{et}.

Using the dataset of Rokicki and Dogon~\cite{dr}, it seems that $S$ is usually much smaller than $V$ for optimal Sidon sets, and that about half of the value of $V$ comes from the first and last $T$ values of $i$, i.e., near the two ends of the Sidon set (see~\cite{summerpaper} for a detailed analysis). In this work, we focus all attention to bounding $V$ near the ends of the Sidon set, and dismiss $S$ by using the trivial bound $S(\Ac,T)\ge0$. Additionally, we will always set $T$ to be some constant times $k^{3/2}$, and the bounds on $V(\Ac, T)$ obtained are of the form $V(\Ac, T)\ge vk^{5/2}-O(k^2)$ for some $v$. With this in mind, we simplify the ETSSE to the following:

\begin{corollary}\label{cor:erdosturan}
Let $\tau>0$ be a constant and suppose for all $k$ and all Sidon sets $\Ac$ with $k$ elements and minimum possible diameter that $V(\Ac, \ceil{\tau k^{3/2}})\ge vk^{5/2}-O(k^2)$. Then
\[ b_\infty\le \tau + \frac{1}{\tau}-\frac{v}{\tau^2}. \]
\end{corollary}
\begin{proof}
Let $T=\ceil{\tau k^{3/2}}$; note $T=\tau k^{3/2} + O(1)$. Let $\Ac=\Ac_k$ to be a Sidon set of $k$ elements with minimum possible diameter. From Theorem~\ref{thm:erdosturan}, we have
\begin{align*}
    \diam(\Ac) &\ge \frac{\tau^2 k^5 + O(k^{7/2})}{\tau^2 k^3 + (\tau-v) k^{5/2} + O(k^2)} - \tau k^{3/2} - O(1) \\
    &= k^2 - \left(\frac{1}{\tau} - \frac{v}{\tau^2}\right)k^{3/2} - \tau k^{3/2} - O(k)
\end{align*}
and the result follows by taking the limit as $k\to \infty$.
\end{proof}

The bounds we obtain on $b_\infty$ do not come from considering a single value of $T$. Indeed, for any particular value of $T$, one can construct a sequence of $\Ac$ with $k\to\infty$ so that the best possible value of $v$ in the corollary above is 0. Instead, we come up with multiple bounds on $v$ for different window sizes $T$, so that if a bound on $v$ is poor for one value of $T$, a different $T$ gives a good bound on $v$.

Many results about finite Sidon sets are presented by bounding $R(n)$ rather than $\diam(\Ac)$. The following proposition is a simple calculation and shows how to convert between the two forms:

\begin{proposition}
We have $R(n)\le n^{1/2} + \frac{c}{2}n^{1/4} + O(1)$ if and only if $\diam(\Ac)\ge k^2 - ck^{3/2} - O(k)$ for all Sidon sets $\Ac$ with $\abs{\Ac} = k$. Also, $R(n)\le n^{1/2} + \frac{c}{2}n^{1/4} + o(n^{1/4})$ is equivalent to $\diam(\Ac)\ge k^2 - ck^{3/2} - o(k^{3/2})$.\qed
\end{proposition}

Erd\H{o}s asked if $R(n) = n^{1/2} + o(n^{1/4})$. By the above proposition, this is equivalent to proving that $b_\infty = 0$, which by a slight modification to Corollary~\ref{cor:erdosturan} could be done if one was able to prove
\[ 2S(\Ac, \tau k^{3/2}) + V(\Ac, \tau k^{3/2}) \ge (\tau^3 + \tau)k^{5/2} \]
for some $\tau$.

\section{Using Two Windows}
\label{sec:twowindow}

The goal in this section is to show:

\begin{theorem}\label{thm:twowindow}
Let $\Ac$ be a finite Sidon set with $k$ elements. Then 
\[\diam(\Ac)\ge k^2 - 1.99058k^{3/2} - O(k).\]
\end{theorem}

The proof will demonstrate most of the main ideas in the later computer-assisted bound $b_\infty\ge 1.96365$.

\begin{proof}
We may assume $\Ac$ is a Sidon set with $k$ elements and minimum possible diameter. Fix a constant $\tau>0$ (to be chosen later) and define $T\coloneqq\ceil{\tau k^{3/2}}$. Also define
\[ \overline{A}\coloneqq\frac{kT}{T+\diam(\Ac)}, \]
and recall the definition $A_i^{(T)} = \abs{\Ac \cap [i-T,i)}$ from the statement of Theorem~\ref{thm:erdosturan}. Note $\overline{A} \sim \tau k^{1/2}$ since $\diam(\Ac)\sim k^2$.

By the ETSSE, if we show that $A_i^{(T)}$ has a large total squared deviation from $\overline{A}$, we obtain a lower bound on $V(\Ac, T)$, which gives us a lower bound on $\diam(\Ac)$. The lower bound on the deviation will come from analyzing the ends of the Sidon set, essentially using the observation that $A_{\min(\Ac)}^{(T)}=A_{\max(\Ac)+T+1}^{(T)}=0\ll \overline{A}$ so $A_i^{(T)}$ deviates greatly from $\overline{A}$ for $i$ near $\min(\Ac)$ and $\max(\Ac)+T$.

The key insight is that if there were only a small deviation between $A_i^{(T)}$ and $\overline{A}$ for $i$ near the ends of $\Ac$, we learn some information about the distribution of elements of $\Ac$: near the ends, there must first be a very high density of elements in order for $\Ac_i^{(T)}$ to increase quickly to get close to $\overline{A}$, followed by a very low-density region to not overshoot $\overline{A}$ too much. This can be exploited by considering a second window size $T'<T$. In particular, there must be a large deviation between $A_i^{(T')}$ and
\[ \overline{A}{}'\coloneqq\frac{kT'}{T'+\diam(\Ac)} \]
for many values of $i$ because $A_i^{(T')}$ will be significantly less than $\overline{A}{}'$ in the low-density region.

Specifically, fix \textit{levels} $\alpha_1$ and $\alpha_2$ with $0<\alpha_1<1<\alpha_2$. Define the following \textit{cutoff points}:
\begin{itemize}
    \item For $j\in\{1,2\}$, let $0\le u_j\le 1$ be such that $\min(\Ac)+u_jT$ is the minimum value of $i\in [\min(\Ac), \min(\Ac)+T]$ such that $A_i^{(T)}\ge \alpha_j\overline{A}$, or $u_j=1$ if there is no such $i$.
    \item For $j\in\{1,2\}$, let $0\le v_j\le 1$ be such that $\max(\Ac)+T+1-v_jT$ is the maximum value of $i\in [\max(\Ac)+1, \max(\Ac)+T+1]$ such that $A_i^{(T)}\ge \alpha_j\overline{A}$, or $v_j=1$ if there is no such $i$.
    \item $w_1\coloneqq u_1+v_1$ and $w_2\coloneqq u_2+v_2$.
\end{itemize}
Now note
\begin{align*}
    V(\Ac, T) &= \sum_{i=\min(\Ac)+1}^{\max(\Ac)+T} \left(A_i^{(T)} - \overline{A}\right)^2 \\
    &\ge \left[\sum_{i=\min(\Ac)+1}^{\min(\Ac)+u_1T}+\sum_{i=\min(\Ac)+u_2T+1}^{\min(\Ac)+T}+\sum_{i=\max(\Ac)+1}^{\max(\Ac)+T-v_2T}+\sum_{i=\max(\Ac)+T+1-v_1T}^{\max(\Ac)+T}\right] \left(A_i^{(T)} - \overline{A}\right)^2 \\
    &\ge (u_1+v_1)T(\alpha_1-1)^2\tau^2k + (2-u_2-v_2)T(\alpha_2-1)^2\tau^2k - O(k^2) \\
    &= (w_1(\alpha_1-1)^2 + (2 - w_2)(\alpha_2-1)^2)\tau^3 k^{5/2} - O(k^2)
\end{align*}
which implies (due to Corollary~\ref{cor:erdosturan})
\begin{equation}
    b_\infty \le \tau + \frac{1}{\tau} - \tau (w_1(\alpha_1-1)^2 + (2 - w_2)(\alpha_2-1)^2).\label{eq:bound1}
\end{equation}

The best bound we can obtain on $\diam(\Ac)$ using \eqref{eq:bound1} is by setting $\tau=1$, and we get $\diam(\Ac)\ge k^2 - 2k^{3/2} + O(k)$ since $u_1,v_1\ge 0$ and $u_2,v_2\le 1$. But if we are near this extreme case of $u_1,v_1,u_2,v_2$, that means that $A_i^{(T)}$ rockets up to $\alpha_1\overline{A}$ quickly and then stays below $\alpha_2\overline{A}$ the next $(u_2-u_1)T$ steps. This means there are at most $(\alpha_2-\alpha_1)\overline{A}$ elements of $\Ac$ in a range of width $(u_2-u_1)T$. By considering a second window size $T'<T$, we will see that $A_i^{(T')}$ is far below $\overline{A}{}'$ for many values of $i$.

Specifically, suppose $T'=\ceil{cT}$ where $0<c<1$. Define $(x)_+\coloneqq\max\{0, x\}$. Suppose $u_2-u_1>c$. Then for $i$ from $\min(\Ac)+u_1T+T'+1$ to $\min(\Ac)+u_2T$, we have $A_i^{(T')}\le (\alpha_2-\alpha_1)\overline{A}$; choosing $c$ carefully, this is less than $\overline{A}{}'$.

Additionally, there are only $\alpha_1\overline{A}$ elements of $\Ac$ from $\min(\Ac)$ to $\min(\Ac)+u_1T$; if $\alpha_1<c$, we have that $A_i^{(T)}$ is less than $\overline{A}{}'$ in this range. Similar analysis holds replacing $u$ with $v$.

Therefore
\begin{align*}
    V(\Ac, T') &= \sum_{i=\min(\Ac)+1}^{\max(\Ac)+T'} \left(A_i^{(T')} - \overline{A}{}'\right)^2 \\
    &\ge \left[\sum_{i=\min(\Ac)+u_1T+1+T'}^{\min(\Ac)+u_2T}+ \sum_{i=\max(\Ac)+T-v_2T+1+T'}^{\max(\Ac)+T-v_1T} \right. \\
    &\qquad\qquad \left.+\sum_{i=\min(\Ac)+1}^{\min(\Ac)+u_1T}+\sum_{i=\max(\Ac)+T+1-v_1T}^{\max(\Ac)+T}\right] \left(A_i^{(T')} - \overline{A}{}'\right)^2 \\
    &\ge ((u_2-u_1-c)_++(v_2-v_1-c)_+)(c-(\alpha_2-\alpha_1))_+^2\tau^3 k^{5/2} \\
    &\qquad\qquad +(u_1+v_1)(c-\alpha_1)_+^2\tau^3k^{5/2} - O(k^2) \\
    &\ge ((w_2-w_1-2c)(c-(\alpha_2-\alpha_1))_+^2+w_1(c-\alpha_1)_+^2)\tau^3 k^{5/2} - O(k^2)
\end{align*}
so
\begin{equation}
    b_\infty \le c\tau + \frac{1}{c\tau} - \frac{\tau}{c^2}((w_2-w_1-2c)(c-(\alpha_2-\alpha_1))_+^2+w_1(c-\alpha_1)_+^2).\label{eq:bound2}
\end{equation}

It remains to choose parameters $\tau, \alpha_1, \alpha_2,$ and $c$, and then show that for any choice of $0\le w_1\le w_2\le 2$, either \eqref{eq:bound1} or \eqref{eq:bound2} gives $\diam(\Ac)\ge k^2 - 1.99058k^{3/2} + O(k)$. We choose the following parameters:
\begin{align*}
    \tau &= 1.07950, \\
    \alpha_1 &= 0.72720, \\
    \alpha_2 &= 1.31609, \\
    c &= 0.86838.
\end{align*}
Note $c>\alpha_2-\alpha_1$ and $c>\alpha_1$. Inequality \eqref{eq:bound1} implies
\[ b_\infty\le 1.7901428 - 0.0803363w_1 + 0.1078559w_2, \]
while inequality \eqref{eq:bound2} implies
\[ b_\infty\le 3.3009719 + 0.7181409w_1 - 0.7466741w_2. \]
(Note that for each of the six coefficients appearing in these two bounds, we have rounded them in the direction that makes the resulting bound weaker, so the loss of precision by truncating to seven digits causes no issue.) 

If
\[ 1.5108291+0.7984772w_1-0.8545300w_2\ge 0 \]
then the first bound is at most $1.99058$, while if the reverse inequality holds, the second bound is at most $1.99058$, completing the proof.
\end{proof}

Here, the parameters $\tau,\alpha_1,\alpha_2,c$ are near a local extremum; nudging any of them by $10^{-5}$ results in a worse bound on $b_\infty$.

\section{More Parameters}
\label{sec:nwindow}

We can improve the bound obtained in the previous section by introducing more parameters; specifically, we will introduce more levels $\alpha_i$ and more windows $T_j'$. Notice that at the end of the proof in the previous section, we obtained two bounds on $b_\infty$, each affine functions in two variables. Each new level $\alpha_i$ introduced will increase the number of variables in each bound by one, and each window $T_j'$ introduced will increase the number of bounds on $b_\infty$ by one. Some complications will arise due to the fact that some of the bounds we derive in this section will only be piecewise affine; this is one of the main reasons we require extensive computer assistance to verify our main result.

\subsection{Generalizing the Bounds}

It will first be convenient to ``symmetrize'' $\Ac$. Define
\[ B_i^{(T)}\coloneqq\frac{A_{\min(\Ac)+i}^{(T)}+A_{\max(\Ac)+T+1-i}^{(T)}}{2}. \]
In other words, let $m=\frac{\min(\Ac)+\max(\Ac)}{2}$ be the midpoint of $\Ac$; then $B_i^{(T)}$ counts the number of elements of $\Ac$ in $[i-T,i)$ plus the number in the reflection of $[i-T,i)$ about $m$ (which is $(2m-i,2-i+T]=[\min(\Ac)+\max(\Ac)+1-i,\min(\Ac)+\max(\Ac)+T+1-i)$), all divided by 2, effectively ``averaging'' the two ends of $\Ac$.

Fix levels $0<\alpha_1<\alpha_2<\dots<\alpha_K$ for some $K$. Also fix $\tau>0$ and let $T\coloneqq\ceil{\tau k^{3/2}}$ and $\overline{A}\coloneqq\frac{kT}{T+\diam(\Ac)}$ as before. For $1\le j\le K$, let $0\le w_j\le 1$ be such that $\min(\Ac)+w_jT$ is the minimum value of $i\in [\min(\Ac), \min(\Ac)+T]$ such that $B_i^{(T)}\ge \alpha_j\overline{A}$, or $w_j=1$ if there is no such $i$. Also define $w_0=\alpha_0\coloneqq0$, $w_{K+1}\coloneqq1$, and $\alpha_{K+1}\coloneqq\infty$. We sometimes write $\alpha=(\alpha_0, \alpha_1, \dots, \alpha_{K+1})$ and $w=(w_0, w_1, \dots, w_{K+1})$. See Figure~\ref{fig:explanation} for a visual representation of the relationship between $\alpha$ and $w$.

\begin{figure}[htb!]
\begin{center}
\begin{tikzpicture}
\draw[thick, ->] (-1, 0) -- (11, 0) node[right] {$i$};
\draw[thick, ->] (0, -1) -- (0, 8) node[above] {$B_i^{(T)}$};
\node[below right] at (0, 0) {$0=w_0T$};
\node[below] at (3, 0) {$w_1T$};
\node[below] at (5, 0) {$w_2T$};
\node[below] at (8, 0) {$w_3T$};
\node[below] at (10, 0) {$T=w_4T$};

\node[above left] at (0, 0) {$0=\alpha_0\overline{A}$};
\node[left] at (0,2) {$\alpha_1\overline{A}$};
\node[left] at (0,3) {$\alpha_2\overline{A}$};
\node[left] at (0,4.5) {$\overline{A}$};
\node[left] at (0,6) {$\alpha_3\overline{A}$};
\node[left] at (0,8) {$\alpha_4\overline{A}=\infty$};

\draw[dashed] (0,2) -- (3,2) -- (3,0);
\draw[dashed] (0,3) -- (5,3) -- (5,0);
\draw[red, dashed] (0,4.5) -- (11,4.5);
\draw[dashed] (0,6) -- (8,6) -- (8,0);
\draw[dashed] (10,7) -- (10,0);

\draw (0,0) to[out=70, in=200] (3,2) to[out=20, in=250] (5,3) to[out=70, in=200] (8,6) to[out=20, in=180] (10,6.5);
\draw[decorate,decoration={snake,amplitude=1.5,segment length=10, post length=0, pre length=0}] (10, 6.5) -- (11, 5.5);
\end{tikzpicture}
\end{center}
\caption{A visual guide to the parameters with $K=3$ assuming $\min(\Ac)=0$.}
\label{fig:explanation}
\end{figure}
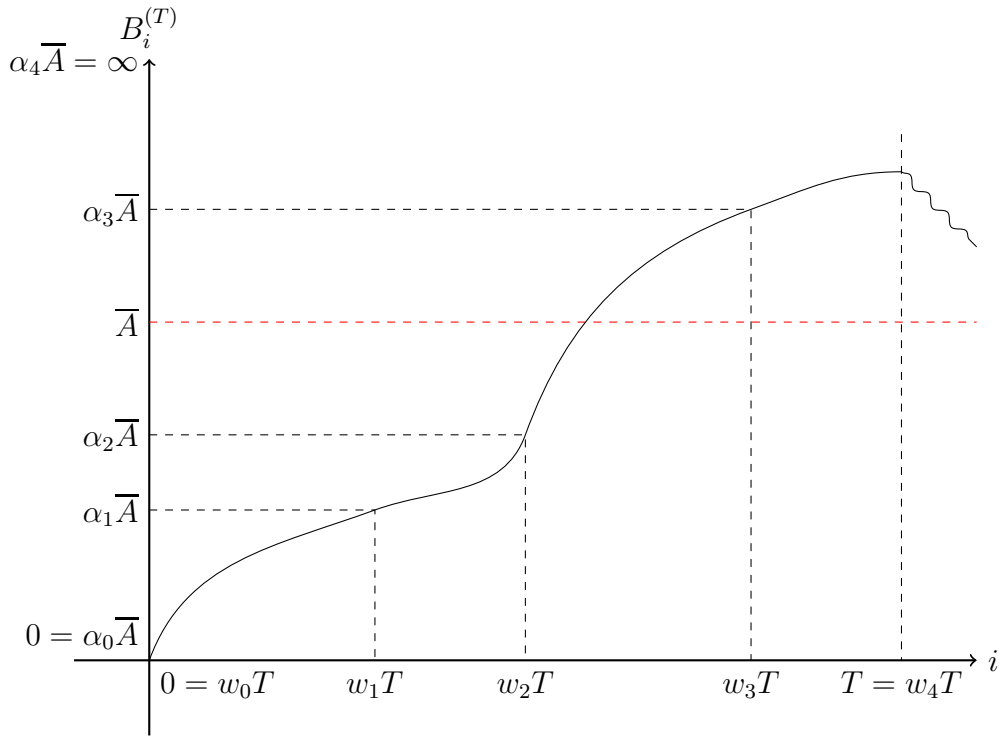

Notice a slight difference in how the cutoffs $w$ are scaled compared to the proof of Theorem~\ref{thm:twowindow}: here, $w_j$ is between 0 and 1 (inclusive), while in that proof, each $u_j$ and $v_j$ was between 0 and 1 and each $w_j$ was between 0 and 2.

Here is the generalization of the first bound appearing in the proof of Theorem~\ref{thm:twowindow}:

\begin{lemma}\label{lem:bound1}
Let $\Ac$ be a finite Sidon set with $k$ elements. With the $\tau$, $\alpha$, and $w$ defined above, we have $\diam(\Ac)\ge k^2 - bk^{3/2}-O(k)$ where
\[ b\le \tau+\frac{1}{\tau}-2\tau\sum_{j=0}^{K}(w_{j+1}-w_j)\min_{\alpha_j\le z \le \alpha_{j+1}}(z-1)^2. \]
\end{lemma}
\begin{proof}
Define
\[ W_j\coloneqq[\min(\Ac)+w_jT+1,\min(\Ac)+w_{j+1}T]. \]
for $0\le j\le K$.

Unless $\alpha_j\le 1\le \alpha_{j+1}$, the value of $B_i^{(T)}$ is bounded away from $\overline{A}$ for $i\in W_j$. Indeed, if $\alpha_j<\alpha_{j+1}<1$, then $\abs{B_i^{(T)}-\overline{A}}\ge (1-\alpha_{j+1})\overline{A}-O(1)$ for $i\in W_j$, and if $1<\alpha_j<\alpha_{j+1}$, then $\abs{B_i^{(T)}-\overline{A}}\ge (\alpha_j-1)\overline{A}-O(1)$ for $i\in W_j$. Put another way, we have
\[ \abs{B_i^{(T)}-\overline{A}}\ge \overline{A}\min_{\alpha_j\le z\le \alpha_{j+1}}\abs{z-1}-O(1). \]
whenever $i\in W_j$. Thus
\begin{align*}
    V(\Ac, T) &= \sum_{i=\min(\Ac)+1}^{\max(\Ac)+T} \left(A_i^{(T)} - \overline{A}\right)^2 \\
    &\ge \left[\sum_{i=\min(\Ac)+1}^{\min(\Ac)+T}+\sum_{i=\max(\Ac)+1}^{\max(\Ac)+T}\right] \left(A_i^{(T)} - \overline{A}\right)^2 \\
    &\ge \sum_{i=\min(\Ac)+1}^{\min(\Ac)+T} 2\left(B_i^{(T)} - \overline{A}\right)^2
\end{align*}
since $(x-z)^2+(y-z)^2\ge 2\left(\frac{x+y}{2}-z\right)^2$ for all $x,y,z$. Continuing,
\begin{align*}
    V(\Ac, T) &\ge 2\sum_{j=0}^K\sum_{i\in W_j} \left(B_i^{(T)} - \overline{A}\right)^2 \\
    &\ge 2\sum_{j=0}^K\abs{W_j}\tau^2 k\min_{\alpha_j\le z\le \alpha_{j+1}}(z-1)^2 - O(k^2) \\
    &= 2\sum_{j=0}^K(w_{j+1}-w_j)\tau^3 k^{5/2}\min_{\alpha_j\le z\le \alpha_{j+1}}(z-1)^2 - O(k^2).
\end{align*}
This gives the desired bound after applying Corollary~\ref{cor:erdosturan}.
\end{proof}
Notice that if $K=2$ and $\alpha_1<1<\alpha_2$, the bound in this lemma exactly matches the first bound obtained in the proof of Theorem~\ref{thm:twowindow}. Now let us generalize the second bound.

Fix $c>0$ and define $T'\coloneqq\ceil{cT}$ and $\overline{A}{}'\coloneqq\frac{kT'}{T'+\diam(\Ac)} = c\overline{A}+O(1)$ as before. Consider some $i\in[\min(\Ac)+1, \min(\Ac)+T]$, and let us attempt to bound $B_i^{(T')}$ away from $\overline{A}{}'$. For each $i$, let $x_i$ be the smallest value of $j$ so that that $\min(\Ac)+w_jT+1\ge i-T'$ and let $y_i$ be the largest value of $j$ so that $\min(\Ac)+w_jT < i$. This means that the entire window $[i-T', i)$ lies inside $I=W_{x_i-1}\cup W_{x_i+1}\cup\dots \cup W_{y_i}$, adopting the definition of $W_j$ from the proof of Lemma~\ref{lem:bound1}. But we know there are only at most $2(\alpha_{y_i}-\alpha_{x_i-1})\overline{A}$ elements of $\Ac$ in $I$ plus the mirror of $I$ about the midpoint of $\Ac$, so
\[ B_i^{(T')}\le (\alpha_{y_i}-\alpha_{x_i-1})\overline{A} + O(1) \]
so
\begin{equation}\label{eq:pre1bound2} \abs{B_i^{(T')} - \overline{A}{}'} \ge (c-(\alpha_{y_i}-\alpha_{x_i-1}))_+\tau k^{1/2} - O(1), \end{equation}
here using the notation $(x)_+=\max\{0,x\}$ from the proof of Theorem~\ref{thm:twowindow}.

At the same time, the window $[i-T', i)$ contains $W_{x_i}\cup\dots\cup W_{y_i-1}$, so there are at least $(\alpha_{y_i-1}-\alpha_{x_i})\overline{A}$ elements of $\Ac$ in this range plus its mirror, so we also have
\begin{equation}\label{eq:pre2bound2} \abs{B_i^{(T')} - \overline{A}{}'} \ge ((\alpha_{y_i-1}-\alpha_{x_i})-c)_+\tau k^{1/2} - O(1). \end{equation}

To make sense of \eqref{eq:pre1bound2} and \eqref{eq:pre2bound2} in the case $x_i$ or $y_i$ is zero, we should additionally define $\alpha_j=0$ if $j<0$. Note $\alpha_{y_i}-\alpha_{x_i-1}\ge \alpha_{y_i-1}-\alpha_{x_i}$, so only one of $(c-(\alpha_{y_i}-\alpha_{x_i-1}))_+$ and $((\alpha_{y_i-1}-\alpha_{x_i})-c)_+$ can be nonzero. Inequalities \eqref{eq:pre1bound2} and \eqref{eq:pre2bound2} can be succinctly combined to
\begin{equation}\label{eq:prebound2}
    \abs{B_i^{(T')} - \overline{A}{}'} \ge \tau k^{1/2}\min_{\alpha_{y_i-1}-\alpha_{x_i}\le z\le \alpha_{y_i}-\alpha_{x_i-1}}\abs{z-c} - O(1)
\end{equation}

Of course, the value of $x_i$ is the same as $x_{i+1}$ everywhere except at the $i$ when $i-T'$ passes $\min(\Ac)+w_jT$ for some $j$; define $q_0<q_1<\dots<q_G$ such that $x_{i+1}$ and $x_i$ are different when $i=\min(\Ac)+q_jT$ for some $j$. Similarly, $y_i$ only changes when $i$ passes $\min(\Ac)+w_jT$ for some $j$; let $r_0<r_1<\dots<r_H$ be such that $y_{i+1}$ and $y_i$ are different when $i=\min(\Ac)+r_jT$ for some $j$.

Now define
\[ P \coloneqq \{x\in \{q_j\}_{j=0}^G \cup \{r_j\}_{j=0}^H\mid 0\le x\le 1\}. \]
Order the elements of $P$ as $p_0<p_1<\dots<p_L$, so each $p_j$ either corresponds to a $q_{j'}$ or an $r_{j'}$. It is the case, due to the definitions of $w_0$ and $w_{K+1}$, that $p_0=r_0=0$ and $p_L=r_H=1$.

For each $1\le j \le L$, define $\zeta_j$ and $\eta_j$ so that for $i\in [\min(\Ac)+1+p_{j-1}T,\min(\Ac)+p_jT]$, $\zeta_j=x_i$ and $\eta_j=y_i$. We write $q,r,p,\zeta,\eta$ to mean $(q_0,\dots,q_G)$, $(r_0,\dots,r_H)$, $(p_0,\dots,p_L)$, $(\zeta_1,\dots,\zeta_L)$, and $(\eta_1,\dots,\eta_L)$, respectively. Here is the generalization of the second bound:

\begin{lemma}\label{lem:bound2}
Let $\Ac$ be a finite Sidon set with $k$ elements. With the previous definitions of $\tau,c,\alpha,p,\zeta,\eta$, we have $\diam(\Ac)\ge k^2-bk^{3/2}-O(k)$ where
\[
    b\le c\tau + \frac{1}{c\tau} -2\frac{\tau}{c^2}\sum_{j=1}^{L}(p_j-p_{j-1})\min_{\alpha_{\eta_j-1}-\alpha_{\zeta_j}\le z\le \alpha_{\eta_j}-\alpha_{\zeta_j-1}}(z-c)^2.
\]
\end{lemma}
\begin{proof}
We have
\begin{align*}
    V(\Ac, T') &\ge \sum_{i=\min(\Ac)+1}^{\min(\Ac)+T} 2\left(B_i^{(T')} - \overline{A}{}'\right)^2 \\
    &= 2\sum_{j=1}^{L}\sum_{i=\min(\Ac)+1+p_{j-1}T}^{\min(\Ac)+p_jT}\left(B_i^{(T')} - \overline{A}{}'\right)^2 \\
    &\ge 2\sum_{j=1}^{L} (p_j-p_{j-1})\tau^3 k^{5/2}\min_{\alpha_{\eta_j-1}-\alpha_{\zeta_j}\le z\le \alpha_{\eta_j}-\alpha_{\zeta_j-1}}(z-c)^2 - O(k^2),
\end{align*}
in the last line using \eqref{eq:prebound2} to bound $\left(B_i^{(T')} - \overline{A}{}'\right)^2$, and using the definitions of $p,\zeta,\eta$ to replace $x_i$ and $y_i$ with $\zeta_j$ and $\eta_j$. The result then follows from Corollary~\ref{cor:erdosturan}.
\end{proof}

Notice that we can actually employ more than two window sizes, simply using multiple instances of Lemma~\ref{lem:bound2} with different values of $c$ (and the appropriate resulting $p,\zeta,\eta$, which depend on $c$), and this will improve the final bound on $b_\infty$. For example, using three windows of sizes $T=\ceil{\tau k^{3/2}}$, $T_1'=\ceil{c_1T}$, and $T_2'=\ceil{c_2T}$, for any values of $w$, we get three bounds on $b$: one from Lemma~\ref{lem:bound1} and two from Lemma~\ref{lem:bound2}, using $c=c_1$ and $c=c_2$.

\subsection{Combining the Bounds}

Combining the bounds from Lemma~\ref{lem:bound1} and Lemma~\ref{lem:bound2} is not so straightforward compared to the proof of Theorem~\ref{thm:twowindow}. The complication arises mostly from the failure of the second bound to be affine in $w$; indeed, the second bound is actually only piecewise affine in the $w$, with the cutoffs between the pieces occurring when some $q_j$ equals some $r_{j'}$; this in turn changes the values of $\zeta$ and $\eta$. We call the domains of the pieces making up the second bound the \textit{cells}, and the affine function on each cell given by the second bound the \textit{cell function}. We will describe how to determine the boundary and cell function of each cell; then, splitting the analysis into a different cases based on the cells will lead to the final bound on $b_\infty$, though there are far too many cases to check by hand. However, the cases can be enumerated programmatically, and each case can be efficiently analyzed using linear programming.

\subsubsection{Determining Cell Boundaries and Cell Functions}

Notice that the value of $q_j$ is precisely $w_j+c$, since $x_i$ changes value when $i-T'$ is equal to $\min(\Ac)+w_jT$; i.e. when $i$ is equal to $\min(\Ac)+(w_j+c)T=:\min(\Ac)+q_jT$. Likewise, $r_j$ is actually equal to $w_j$, and in fact $H=K+1$ and $r_H=w_{K+1}=1$. Thus, some $q_j$ equals some $r_{j'}$ precisely when some $w_j+c$ equals some $w_{j'}$. This is an affine constraint on $w$, so the cells are convex polytopes.

Let $S$ be the set of the ways to ``interlace'' $q$ and $r$; specifically, each element of $S$ is a tuple $s=(s_0,\dots,s_{K+1})$ such that for each $0\le j\le K+1$, $r_{s_j}\le q_j$, and $s_j$ is chosen to be the maximum value that this is true. Each element of $S$ potentially corresponds to a cell, though the resulting inequality constraints may have no solution, so some elements of $S$ correspond to empty cells and may later be discarded.

We will illustrate how to find the bounding inequalities for a particular cell as a representative example. For example, if $K=3$ and $s=(0,2,2,4,4)$ (so $s_0=0$, $s_1=2$, and so on), then the corresponding cell has 
\begin{align*}
    p_0 &= r_0 = w_0 = 0 \\
    p_1 &= q_0 = w_0 + c = c \\
    p_2 &= r_1 = w_1 \\
    p_3 &= r_2 = w_2 \\
    p_4 &= q_1 = w_1+c \\
    p_5 &= q_2 = w_2+c \\
    p_6 &= r_3 = w_3 \\
    p_7 &= r_4 = w_4 = 1.
\end{align*}
Here we have determined $L=7$, but the specific value of $L$ depends on $s$. For example, if $s_3$ was $3$ instead of $4$, then $L$ would be $8$, with $p_7=w_3+c$ and $p_8=w_4$ instead. Notice in the case illustrated above that $q_3=w_3+c$ and $q_4=w_4+c$ do not appear in the elements of $P$ since they are at least $r_4=1$, and $P$ contains only elements between 0 and 1.

We need $p_0\le p_1\le \dots \le p_L$, so in this case we have constraints
\begin{align*}
    w_1 &\ge c & \text{from $p_2\ge p_1$}, \\
    w_1 + c &\ge w_2 & \text{from $p_4\ge p_3$}, \\
    w_3 &\ge w_2+c & \text{from $p_6\ge p_5$},
\end{align*}
as well as $0=w_0\le w_1\le \dots \le w_{K+1}=1$, which are always constraints regardless of $s$. Additionally, we have the constraint $w_3+c\ge w_4$ since $s_3=4$. This completes the determination of the cell boundary.

Not all possible values of $s$ correspond to nonempty cells. For example, if $s=(0,0,\dots)$, then we would have to have $q_1=w_1+c\le r_1=w_1$, which is not possible since we will take $c>0$. Generalizing this, we must have $s_j\ge j$ for all $j$ in order for the corresponding cell to be nonempty. The values of $s_j$ must also obviously be weakly increasing in $j$ for the cell to be nonempty. There are even more constraints on $s$ depending on the specific value of $c$; for example, if $c>1/2$, then it is not possible for $q_0=c\le r_1=w_1\le q_1=w_1+c\le r_{K+1}=1$, since here we have $q_1\ge 2c>1$. We can quickly check if a particular $s$ corresponds to a nonempty cell for a given value of $c$ by testing the feasibility of the linear program with the relevant set of inequalities (and an arbitrary objective function).

Given $s$, we can also determine the cell function for the corresponding cell. To do this, we first determine the values of $\zeta$ and $\eta$. We have that $\zeta_j$ is equal to the smallest $j'$ such that $q_{j'}\ge p_j$, and $\eta_j$ is equal to the smallest $j'$ such that $r_{j'}\ge p_j$. Continuing the example cell from before, we have $\zeta=(0,1,1,1,2,3,3)$ and $\eta=(1,1,2,3,3,3,4)$; here $\zeta$ and $\eta$ are 1-indexed. Having determined $\zeta$ and $\eta$, given $\tau$, $c$, and $\alpha_j$, it is easy to rewrite the second bound in the form
\[ a_1^{(s)}w_1 + \dots + a_K^{(s)}w_K + a_{K+1}^{(s)} \]
for some real numbers $a_j^{(s)}$. Thus we can determine the boundary and cell functions for each cell and write the bound from Lemma~\ref{lem:bound2} as a piecewise affine function.

\subsubsection{Optimizing Over Multiple Piecewise Affine Functions}

Now to combine the bounds from Lemma~\ref{lem:bound1} and Lemma~\ref{lem:bound2}, break into two cases for each cell. In the first case, assume that the bound from Lemma~\ref{lem:bound1} is greater than Lemma~\ref{lem:bound2}. Then we are searching for the smallest value of the first bound among the $w$ lying in the cell that make the first bound greater than the second. This is easily written as a linear programming instance: the objective function is the first bound, and the constraints are the constraints imposed by the cell boundary, plus the constraint that the first bound is at least the second; the objective function and all constraints are affine functions.

The second case in this cell is similar, but instead assuming that the second bound is greater than the first, so the objective function is the second bound and the constraints are the cell boundary plus thte constraint that the second bound is greater than the first; again, the objective function and all constraints are affine functions. It may be the case in some cells that either the first bound or second bound is always the larger one, in which case the linear program resulting from the other case will be infeasible, and that case can be discarded. The resulting bound on $b_\infty$ is the maximum over all cases of the solution to the linear program corresponding to that case.

If we employ more than two windows, we are to combine the bound from Lemma~\ref{lem:bound1} with several bounds from Lemma~\ref{lem:bound2}. Each instance of Lemma~\ref{lem:bound2} comes with its own set of cells. Now, we split into even more cases: one case for each bound being the largest over each nonempty intersection formed from choosing one cell from each instance of Lemma~\ref{lem:bound2}. The objective function in a particular case is one of the bounds, and the constraints are the constraints from the cell boundaries of each cell included in the intersection, plus constraints to ensure that the chosen objective function is at least as large as all of the other bounds.

Putting it all together, we have:

\begin{theorem}
If $\Ac$ is a Sidon set with $k$ elements, then $\diam(\Ac)\le k^2 - 1.96365k^{3/2} - O(k)$.
\end{theorem}
\begin{proof}
We choose $K=6$ and employ four windows in total, using one instance of Lemma~\ref{lem:bound1} and three instances of Lemma~\ref{lem:bound2}. Here are the parameters:

\hspace*{\fill}\def\arraystretch{0}
\begin{tabular}{p{0.25\linewidth} p{0.25\linewidth}}
    {\begin{align*}
    \tau &= 1.12733 \\ \\
    c_1 &= 0.66461 \\
    c_2 &= 0.67780 \\
    c_3 &= 0.71884
\end{align*}} & {\begin{align*}
    \alpha_1 &= 0.70749 \\
    \alpha_2 &= 0.78822 \\
    \alpha_3 &= 0.87175 \\
    \alpha_4 &= 1.12464 \\
    \alpha_5 &= 1.18020 \\
    \alpha_6 &= 1.24610
\end{align*}}
\end{tabular}\hspace*{\fill}

The strategy described in this subsection was implemented in Python, using SciPy \cite{scipy} to solve the relevant linear programs. Code is available at \url{https://github.com/dcartermath/sidon}; we invite the reader to look at the code comments for implementation details. It takes approximately 2 minutes to run on a laptop with an Intel i7-10750H CPU. Each instance of Lemma~\ref{lem:bound2} has 127 nonempty cells. There are a total of 24822 triples of cells that have nonempty intersection, leading to 40964 cases in total (after discarding cases that lead to infeasible linear programs). A worst case value of $w$ is $w_1\approx 0.13398$, $w_2\approx 0.30015$, $w_3\approx 0.46220$, $w_4\approx 0.96476$, $w_5\approx 0.97795$, and $w_6=1$. At this point, all four bounds from Lemmas~\ref{lem:bound1} and~\ref{lem:bound2} are equal (up to rounding) to $1.963645$.
\end{proof}

The parameters chosen in the proof above are a local minimum; nudging any of them by $10^{-5}$ leads to a worse bound. However, there may be better local minimum that we did not find. Some effort was made to explore the parameter space and find a good local minimum. Additionally, increasing the number of parameters (particularly increasing $K$) will almost certainly lead to a better bound on $b_\infty$, at the cost of an exponential increase in the time it takes to verify the proof. Some experiments with bounds resulting from smaller $K$ suggest that one may be able to obtain $b_\infty \le 1.95$ with larger $K$, but we were unable to effectively search for good parameters due to the rapidly increasing computational costs.

\section{\texorpdfstring{$g$-Thin Sidon Sets}{g-Thin Sidon Sets}}
\label{sec:thin}

Our methods may also be used to obtain bounds on the diameters of $g$-thin Sidon sets. Recall that a set of integers $\Ac$ is said to be a \textit{$g$-thin Sidon set} if $\abs{\{(a,b)\in \Ac^2\mid a\ne b, a-b=d\}}\le g$ for all $d$. Sidon sets are the same as 1-thin Sidon sets, and $g$-thin Sidon sets are also called $g$-Golomb rulers.

First, we generalize Theorem~\ref{thm:erdosturan} to the case of $g$-thin Sidon sets:

\begin{theorem}\label{thm:thinerdosturan}
    Let $\Ac$ be a finite $g$-thin Sidon set of integers, and $T$ be a positive integer. Let $D_s$ be the subset of integers in $[1,T)$ with exactly $s$ representations as a difference of elements of $\Ac$. Then
    \[
        \diam(\Ac) = \frac{\abs{\Ac}^2 T^2}{gT(T-1)+\abs{\Ac}T-(2S_g(\Ac,T)+V(\Ac,T))}-T,
    \]
    where $A_i^{(T)} \coloneqq \abs{\Ac \cap [i-T,i)}$,
    \[ S_g(\Ac,T) \coloneqq \sum_{s=0}^{g-1} \sum_{r\in D_s} (g-s)(T-r), \]
    and
    \[ V(\Ac,T) \coloneqq\sum_{i=\min(\Ac)+1}^{T+\max(\Ac)} \left( A_i^{(T)} - \frac{\abs{\Ac} T}{T+\diam(\Ac)}\right)^2. \]
\end{theorem}

\begin{proof}
Let $\Ac=\{a_i\}_{i=1}^k$, with $a_1<\dots<a_k$. We may assume that $a_1=0$. Let $A_i \coloneqq \abs{\Ac\cap (-\infty, i)}$. As in the proof of Theorem~\ref{thm:erdosturan} appearing in \cite{summerpaper}, we have
\[\sum_{i=1}^{a_k+T} \binom{A_i}{2}= \frac{1}{2} \frac{k^2T^2}{a_k+T} -\frac{1}{2} k T + \frac{1}{2} V(\Ac,T).\]

Fix a difference $r$. Each pair $(x,y)\in \Ac \times \Ac$ with $y-x=r$ contributes to $\binom{A_i}{2}$ for $T-r$ values of $i$. If $r\in D_s$, then there are $s$ such pairs, and so
\begin{align*}
    \sum_{i=1}^{a_k+T} \binom{A_i}{2} 
    &= \sum_{s=0}^g \sum_{r\in D_s} s (T-r) \\
    &= \sum_{s=0}^g \sum_{r=1}^{T-1} (g+(s-g))(T-r)\1_{r\in D_s} \\
    &= \sum_{r=1}^{T-1} \sum_{s=0}^g g(T-r)\1_{r\in D_s} 
        - \sum_{s=0}^{g-1} \sum_{r\in D_s} (g-s)(T-r) \\
    &= g \binom{T}{2} - \sum_{s=0}^{g-1} \sum_{r\in D_s} (g-s)(T-r) \\
    &= g \binom{T}{2} - S_g(\Ac,T).
\end{align*}
Comparing the two expressions for $\sum \binom{A_i}{2}$ yields the claimed equality.
\end{proof}

Recall that Lemmas~\ref{lem:bound1} and~\ref{lem:bound2} were proved by finding a lower bound on $V(\Ac, T)$. Additionally, the proofs of these lemmas never used the fact that $\Ac$ was a Sidon set to obtain the bound on $V(\Ac, T)$; the Sidon set property was only used to apply Theorem~\ref{thm:erdosturan}. Since the definition of $V(\Ac, T)$ is the same in Theorem~\ref{thm:thinerdosturan} as it was in Theorem~\ref{thm:erdosturan}, the definitions of $w_j$, $p_j$, $\zeta_j$, and $\eta_j$ in Section~\ref{sec:nwindow} still make sense in the $g$-thin case, and Lemmas~\ref{lem:bound1} and~\ref{lem:bound2} generalize immediately to the following:

\begin{lemma}\label{lem:thinbound1}
Let $\Ac$ be a finite $g$-thin Sidon set with $k$ elements. Fix constants $\tau$ and $0=\alpha_0<\alpha_1<\dots<\alpha_K<\alpha_{K+1}=\infty$. Define $w$ as in Section~\ref{sec:nwindow}. We have $\diam(\Ac)\ge k^2/g - bk^{3/2}-O(k)$ where
\[ b\le \tau+\frac{1}{\tau g^2}-2\frac{\tau}{g^2}\sum_{j=0}^{K}(w_{j+1}-w_j)\min_{\alpha_j\le z\le \alpha_{j+1}}(z-1)^2. \]
\end{lemma}

\begin{lemma}\label{lem:thinbound2}
Let $\Ac$ be a finite $g$-thin Sidon set with $k$ elements. Fix constants $\tau$, $c$, and $0=\alpha_0<\alpha_1<\dots<\alpha_K<\alpha_{K+1}=\infty$. Define $p$, $L$, $\zeta$, and $\eta$ as in Section~\ref{sec:nwindow}. We have $\diam(\Ac)\ge k^2/g-bk^{3/2}-O(k)$ where
\[
    b\le c\tau + \frac{1}{c\tau g^2} -2\frac{\tau}{c^2g^2}\sum_{j=1}^{L}(p_j-p_{j-1})\min_{\alpha_{\eta_j-1}-\alpha_{\zeta_j}\le z\le \alpha_{\eta_j}-\alpha_{\zeta_j-1}}(z-c)^2.
\]
\end{lemma}
In these bounds, the error term $O(k)$ is for fixed $g$ and $k\to\infty$.

Similarly to the case of 1-thin Sidon sets, let $\Ac_k^{(g)}$ be a $g$-thin Sidon set with $k$ elements and the minimum possible diameter, and define
\[ b_\infty^{(g)} \coloneqq \limsup_{k\to \infty}\frac{k^2 - \diam(\Ac_k^{(g)})}{k^{3/2}}. \]

For fixed $g$, one can choose parameters $\tau$, $\alpha$, etc. and apply Lemmas~\ref{lem:thinbound1} and~\ref{lem:thinbound2} to find an upper bound on $b_\infty^{(g)}$. Sadly, it seems that locally optimal parameters for a particular $g$ are not related in any simple way to locally optimal parameters for other $g$. Still, we can prove the following:

\begin{theorem}
For any positive integer $g$, there exists $\varepsilon_g>0$ such that $b_\infty^{(g)}\le \frac{2-\varepsilon_g}{g}$.
\end{theorem}
This result was first stated in an equivalent form in~\cite{bfr} (Theorem 6.2), though a detailed proof was omitted.
\begin{proof}
The case $g=1$ has, of course, already been dealt with. Fix $g\ge 2$, and with foresight set
\[ \varepsilon = \frac{100g^2-7g-89}{50g^2(50g^2-49)}. \]
Let $\Ac$ be a $g$-thin Sidon set with $k$ elements. Set
\begin{align*}
    \tau &= 1/g \\
    \alpha_1 &= 0.8 \\
    \alpha_2 &= 1.2 \\
    c &= 25 g^2\varepsilon = \frac{100g^2-7g-89}{2(50g^2-49)}.
\end{align*}
Note $c<1$ since $g\ge 2$; also $c>\alpha_1$. From Lemma~\ref{lem:thinbound1}, we know $\diam(\Ac)\ge k^2/g-bk^{3/2}+O(k)$ where
\[ b \le \frac{2}{g} - \frac{2}{25g^3}(1+w_1-w_2). \]
If $w_1$ is at least $c/2$ or $w_2$ is at most $1-c/2$, then this will give us $b\le \frac{2-\varepsilon}{g}$. In any other case, we will apply Lemma~\ref{lem:thinbound2}. Since we have $0\le w_1\le c\le w_1+c\le w_2\le 1\le w_2+c$, we lie in the cell corresponding to $s=(1,1,3,3)$. We have $p=(0, w_1, c, w_1+c, w_2, 1)$ (0-indexed), $\eta=(1,2,2,2,3)$ (1-indexed), and $\zeta=(0,0,1,2,2)$ (1-indexed), so we know
\begin{align*}
    b &\le \left(c+\frac{1}{c}\right)\frac{1}{g} - \frac{2}{c^2g^3}(c-w_1)(0.8-c)^2 \\
    &\le \left(c+\frac{1}{c}\right)\frac{1}{g} - \frac{1}{cg^3}(0.8-c)^2 \\
    &= \frac{500000g^{6}-35000g^{5}-943775g^{4}+38150g^{3}+447500g^{2}-3710g-2809}{50\left(50g^{2}-49\right)\left(100g^{2}-7g-89\right)g^{3}}.
\end{align*}
Call the right-hand side of the above inequality $X$. Then finally note that
\begin{align*}
    \frac{2-\varepsilon/2}{g} - X &= \frac{151g^{2}-126g+47}{100g^{3}\left(100g^{2}-7g-89\right)} \\
    &\ge 0
\end{align*}
since $g\ge 2$. This completes the proof, with $\varepsilon_g=\varepsilon/2$.
\end{proof}
This proof shows we may take $\varepsilon_g=\Omega(g^{-2})$; in particular, $\varepsilon_g\ge \frac{1}{50g^2}$. There some flexibility in the proof above, so the constant $1/50$ can be optimized, but it seems our methods cannot improve the exponent on $g$. This is because the bounds on $V(\Ac, T)$ we obtain only impact the bound on $b_\infty^{(g)}$ by $\Theta(g^{-3})$, while the choice of $\tau$ and $c$ impacts the bound by $\Theta(g^{-1})$. It is no coincidence, then, that we take $\tau\sim 1/g$ and $c\sim 1$ in this proof; in the limit, only this choice allows both Lemma~\ref{lem:thinbound1} and Lemma~\ref{lem:thinbound2} to give $b_\infty^{(g)}\le 2g^{-1}+o(g^{-1})$. Is it possible that another method can do asymptotically better, say to $\varepsilon_g=\Omega(g^{-1.999})$?

\printbibliography
    
\end{document}